\documentclass[12pt]{amsart}
\usepackage[hmargin=1in,vmargin=1in]{geometry}
\usepackage{url}

\usepackage[numbers,sort&compress]{natbib}
\usepackage{hyperref}
\usepackage[utf8]{inputenc}

\usepackage{hyperref}
\usepackage{amsmath,amssymb}
\usepackage{graphicx}
\usepackage{subfig}

\usepackage{graphicx}
\usepackage{tikz}

\newcommand{\Z}{\mathbb{Z}}

\newcommand{\RR}{\mathbb{R}}

\newcommand{\CC}{\mathbb{C}}
\newcommand{\bx}{{\bf x}}
\newcommand{\bb}{{\bf b}}

\newcommand{\bt}{{\bf t}}
\newcommand{\ba}{{\bf a}}
\newcommand{\bu}{{\bf u}}
\newcommand{\bv}{{\bf v}}

\newcommand{\bp}{{\bf p}}

\newcommand{\mI}{\mathcal{I}}
\newcommand\cI{\mathcal{I}}
\newcommand\cC{\mathcal{C}}
\newcommand\cM{\mathcal{M}}
\newcommand\cD{\mathcal{D}}

\newcommand\cB{\mathcal{B}}

\DeclareMathOperator{\Res}{Res}

\newcommand{\CM}{\mathrm{CM}}

\newtheorem{thm}{Theorem}      % One counter for the monthly
\newtheorem{cor}[thm]{Corollary}     % Numbered along with thm
\theoremstyle{definition} 
\newtheorem{defin}[thm]{Definition}   % Numbered along with thm
\newtheorem{ex}[thm]{Example}        % Numbered along with thm

\begin{document}
\author{Zvi Rosen}
\author{Jessica Sidman}
\author{Louis Theran}
\address{Florida Atlantic University, Boca Raton, FL 33431}
\email{rosenz@fau.edu}
\address{Mount Holyoke College, South Hadley, MA 01075}
\email{jsidman@mtholyoke.edu}
\address{University of St. Andrews, St. Andrews, Scotland}
\email{lst6@st-and.ac.uk}
                               
\title{Algebraic matroids in action}

\begin{abstract}
In recent years, various notions of algebraic independence have
emerged as a central and unifying theme in a number of areas of applied mathematics, 
including algebraic statistics and the rigidity theory of bar-and-joint frameworks.
In each of these settings the fundamental problem is to determine the
extent to which certain unknowns depend algebraically on given data.
This has, in turn, led to a resurgence of interest in algebraic matroids, 
which are the combinatorial formalism for algebraic (in)dependence.  
We give a self-contained introduction to algebraic matroids together with examples 
highlighting their potential application.
\end{abstract}
\maketitle

\section{Introduction.}
Linear independence is a concept that pervades mathematics and applications, but the
corresponding notion of algebraic independence in its various guises is less
well studied.  As noted in the article of Brylawski and Kelly \cite{brylawskiKelly}, between the 1930 and 1937 editions
of the textbook {\em Moderne Algebra} \cite{vdW}, van der Waerden changed his
treatment of  algebraic independence in a field extension to emphasize how the theory
exactly parallels what is true for linear independence in a vector space, showing the
influence of Whitney's \cite{whitney} introduction of of matroids in the intervening years. 
Though var der Waerden did not use the language of matroids, his observations are the foundation
for the standard definition of an algebraic matroid. In this article, we focus on an
equivalent definition in terms of polynomial ideals that is currently useful in
applied algebraic geometry, providing explicit proofs for results that seem to be
folklore.   We highlight computational aspects that tie the 19th century notion of
elimination via resultants to the axiomatization of independence from the early 20th
century to current applications.

We begin by discussing two examples that will illustrate the scope and 
applicability of the general theory. Our intention is that they are different enough
to illustrate the kinds of connections among disparate areas of active mathematical
inquiry that motivated Rota  \cite{kung} to write in 1986 that ``[i]t is as if one were
to condense all trends of present day mathematics onto a single finite structure, a
feat that anyone would \emph{a priori} deem impossible, were it not for the mere
fact that matroids exist.'' 

Our first example is an instance of the \emph{matrix completion} problem in
statistics, chosen to be small enough that we can work out the mathematics by hand. 
In this scenario, a partially filled matrix $M$ of data is given, and a rank $r$ is
specified.  We seek to understand whether we can fill in (or ``complete'') 
the missing entries so that the resulting matrix has rank $r$.  This is
related to how interdependent entries of a matrix are.

\begin{ex}\label{ex: matrix}
Suppose that we are given four  entries of the following  $2 \times 3$ matrix:
\[ \begin{pmatrix}1 & 2 & *\\ * & 6 & 3 \end{pmatrix}.\]   
In how many ways can we fill in the missing entries, shown as $*$, 
if the matrix is to have rank one?
 
To solve this problem, we let 
$M = \begin{pmatrix} a & b & c\\d& e & f \end{pmatrix}$ be a matrix with
indeterminates as entries. If the matrix $M$ has rank one, then all $2 \times 2$
minors are equal to zero:
\[
\begin{array}{llllll}
(1)&  ae-bd = 0, & (2) &  af-cd =0,& (3) & bf-ce=0.
\end{array}
\]
Since $b$ and $e$ are nonzero we can solve equations (1) and (3) for $c$ and $d$.  
We obtain $c = \frac{bf}{e}=1$ and $d = \frac{ae}{b}=3$. Here, 
Equation (2) is a consequence of the others:
\[
af - cd = af - \frac{bf}{e}\frac{ae}{b}=0.
\]

Note that if we  choose values of $a,b,e,$ and $f$ independently, and 
our choices are sufficiently generic (in this case, $b$ and $e$ nonzero suffices), 
we can complete the matrix.  However, the values of $c$ and $d$ depend on the four entries that are already specified, and the rank one completion is unique. In the language of algebraic matroids, 
$\{a,b,e,f\}$ is a maximal independent set of entries in a rank one
$2 \times 3$ matrix. However, not all subsets of four entries are independent, as the $2 \times 2$ minors are algebraic dependence relations.  Indeed, if $\{a,b,d,e\}$ are chosen generically, they will not satisfy Equation (1).
\end{ex}

A similar setup appears in distance geometry, where the fundamental question is to
determine if a list of positive real numbers could represent pairwise distances among
a set of $n$ points in $\RR^d$.

\begin{ex}\label{ex: rigidity} 
Let $G$ be a graph on vertices $\{1, \ldots, n\}$ with
nonnegative edge weights $\ell_{ij}.$ If the $\ell_{ij}$ represent squared distances between
points in $\RR^d,$ they must satisfy various inequalities (e.g., 
they must be nonnegative and satisfy the triangle inequality) 
as well as polynomial relations.

We examine the simplest case, where $\ell_{12}, \ell_{13}, \ell_{23}$ are the (squared)
pairwise distances between three points. There are no polynomial conditions on the lengths
of the edges of a triangle in dimensions $d\ge 2$. However, if the three points lie on a
line, then the area of the triangle with these vertices must be zero. The (squared) area of
a triangle in terms of its edges is given by the classical Heron formula:
\[ 
    A^2 = s(s - \sqrt{\ell_{12}}) (s - \sqrt{\ell_{13}})(s - \sqrt{\ell_{23}}), 
\] 
where $s = \frac{1}{2}(\sqrt{\ell_{12}} + \sqrt{\ell_{13}} + \sqrt{\ell_{23}})$.

The quantity $A^2$ may also be computed by taking $\frac1{16}\det M_3$, where 
\[
    M_3 = \begin{pmatrix}
    2 \ell_{13} & \ell_{13} + \ell_{23} - \ell_{12} \\
    \ell_{13} + \ell_{23} - \ell_{12} & 2 \ell_{23}
    \end{pmatrix}.
\]
Hence, in dimension $d=1$, the squared edge lengths of a 
triangle must satisfy the polynomial relation $\det M_3 = 0$.
The matrix $M_3$ is two times the \textit{Gram matrix} 
of pairwise dot products among the vectors $\bv_1 := \bp_1 - \bp_3$
and $\bv_2 := \bp_2 - \bp_3$, where the $\bp_1,\bp_2,\bp_3$ are unknown points in 
$\RR^d$, which we can check via the computation
\[
    \ell_{12} = \|\bp_1 - \bp_2\|^2 = (\bv_1 - \bv_2)\cdot (\bv_1 - \bv_2) =
    \bv_1\cdot \bv_1 + \bv_2\cdot \bv_2 - 2\bv_1\cdot \bv_2 =  
    \ell_{13} + \ell_{23} - 2\bv_1\cdot \bv_2
\]
This derivation, due to Schoenberg \cite{schoenberg} and Young and Householder \cite{youngHouseholder}, works for more points
(the Gram matrix is $(n-1)\times (n-1)$ for $n$ points) and
any dimension (the Gram matrix of 
point set with $d$-dimensional affine span has rank $d$).  
A related classical construction, the Cayley--Menger matrix,
is due to Menger \cite{menger}.

What we see is that the polynomial relations constraining squared distances of a
$d$-dimensional point set are all derived from the $(d+1)\times (d+1)$ minors of a
Gram matrix. These polynomial relations govern how independently the interpoint
distances may be chosen. For example, we see that if three points are collinear, then 
we are free to choose two of the interpoint distances in any way.
Once these are chosen, there are (at most) two possibilities 
for the third.
\end{ex}

At their core, the questions that we ask in Examples \ref{ex: matrix} and \ref{ex:
rigidity} are about trying to determine to what extent certain unknown values
(distances or matrix entries), are independent of the known ones. Matroids provide a
combinatorial abstraction for the study of independence. This perspective was brought
to distance geometry by Lov\'asz and Yemini \cite{lovasz}. The point of view there is
that the Jacobian of distance constraints defines a linear matroid; by analogy, a
similar idea applies to matrix completion in work of Singer and Cucuringu  \cite{singer}.

Recently, work on problems like these has focused on the fact that the matroids
appearing are \textit{algebraic}. In addition to dependent sets we also have the
specific polynomials witnessing the dependencies. This aspect of algebraic matroids
has been understood for some time, going back to Dress and Lov\'asz in \cite{dressLovasz}, 
actually exploiting them in applications seems to be newer (see 
\cite{kiralyTheran2,kiralyTheran,grossSullivant}).

Notions of independence abound in other applications as well. For example, chemical
reaction networks with mass-action dynamics can be described by a polynomial system of
ODE's. The algebraic properties of these systems at steady state were first exploited by 
Gatermann \cite{gatermann} and further developed by 
Craciun, Dickenstein, Shiu and Sturmfels \cite{craciun}. 
If a chemist
identifies an algebraically dependent set of variables, then she can perform
experiments to determine whether the corresponding substances are related
experimentally. These dependence relations on subsets, along with their algebraic
properties, were used by Gross, Harrington, Rosen and Sturmfels \cite{grossHarrington} to simplify computations.

\subsection{Guide to reading}
The sequel is structured as follows.  We first 
 briefly recall the general definition of a
matroid. In the subsequent sections we will discuss three ways of defining algebraic
matroids: via a prime ideal, an algebraic variety, or a field extension. 
Historically, the latter was the standard definition, but the first two are 
more natural in modern applications.  We will see that all three definitions are
equivalent, and that there are canonical ways to move between them.  We then conclude by
revisiting the applications disucssed in the introduction in more detail.

\section{Matroids: axiomatizing (in)dependence.} \label{sec: matroid axiom}

The original definition of a matroid is by Whitney \cite{whitney}, who wanted to
simultaneously capture notions of independence in linear algebra and graph theory. The
terminology, with ``bases'' borrowed from linear algebra and ``circuits'' from graph
theory, reflects these origins. It is not surprising that contemporaneous
mathematicians such as van der Waerden, Birkhoff, and Maclane were also drawn into
this circle of ideas. As Kung writes in \cite{kung},
\begin{quotation}    
It was natural, in a decade when the axiomatic method was still a fresh idea, to
attempt to find the fundamental properties of dependence common to these notions,
postulate them as axioms, and derive their common properties from the axioms in a
purely axiomatic manner.
\end{quotation}
We present these axioms in this section.

\begin{defin}\label{def: independent}
	A \emph{matroid} $(E,\cI)$ is a pair where $E$ is a 
	finite set and $\cI\subseteq 2^E$ satisfies
	\begin{enumerate}
		\item $\emptyset\in \cI$
		\item If $I_2\subseteq I_1\in \cI$, then 
		$I_2\in \cI$.
		\item If $I_1$ and $I_2$ are in $\cI$ 
		and $|I_2| > |I_1|$, there is $x\in I_2\setminus I_1$
		so that $I_1\cup \{x\}\in \cI$.
	\end{enumerate}
The sets $I\in \cI$ are called \emph{independent}.
    
The complement of $\cI$ is denoted $\cD,$  the \emph{dependent sets}.
The subset $\cC \subseteq \cD$ of inclusion-wise minimal dependent sets is the set of  
\emph{circuits} of the matroid. Finally, $\cB \subseteq \cI$ of maximal independent sets is 
the set of \emph{bases} of $(E,\cI)$. The bases are all the same size, 
which is called the \emph{rank} of the matroid; 
more generally, the rank of a subset $A\subseteq E$ is the
maximum size of an independent subset of $A$.
\end{defin}
Intuitively, independence should be preserved by taking subsets, and this gives the
motivation for the first two axioms. For the last axiom (augmentation), recall that in
linear algebra any linearly independent set of vectors can always be augmented with
some vector from a larger linearly independent set without creating a dependence.

As the name suggests, a ``matroid'' is an abstract version of a matrix, and every
matrix gives rise to a matroid. If 
$M =  (\bx_1\, \cdots\, \bx_n)$ is an $m\times n$
matrix with columns $\bx_i \in \RR^m$, we define $\cI_M$ to be the set of all $I
\subseteq [n]$ with $\{\bx_i \mid i \in I\}$ linearly independent. The reader may check
that the axioms are satisfied in Example \ref{ex: vec config} by inspection and the
verification in general is a simple linear algebra exercise.

\begin{ex}\label{ex: vec config}
    Let 
    \[
      A=  \begin{pmatrix}
        1 & 1 & 1 & 0 & 0 & 0\\
        0 & 0 & 0 & 1 & 1 & 1\\
        1 & 0 & 0 & 1 & 0 & 0\\
        0 & 1 & 0 & 0 & 1 & 0\\
        0 & 0 & 1 & 0 & 0 & 1
        \end{pmatrix}.
    \]
If we label the columns $a, \ldots, f$ from right to left, then we can see that
the columns with labels $\{a,b,e,f\}$ form a basis while the columns $\{a,b,d,e\}$
form a circuit. In fact, the column vectors of $A$ all satisfy the same
dependencies as the entries of $M$ in Example \ref{ex: matrix}.  We will see later 
that this is not an accident.
\end{ex}

It is natural to ask if every matroid arises from a matrix in this way. Whitney posed
this question in his foundational paper \cite{whitney} where he proposed that the
matroid on seven elements of the Fano projective plane whose circuits are depicted in
Figure \ref{fig: fano} was a ``matroid with no corresponding matrix.'' However,
Whitney's proof does not hold in characteristic 2 and indeed there is a $3 \times 7$
matrix with entries in $\mathbb{F}_2$ representing this matroid. Whitney was quite
aware of this, but in his language, a matrix meant a matrix with complex entries.

\begin{figure}[htbp]
	\centering
	\subfloat[][Fano matroid]{\includegraphics[height=0.32\textwidth]{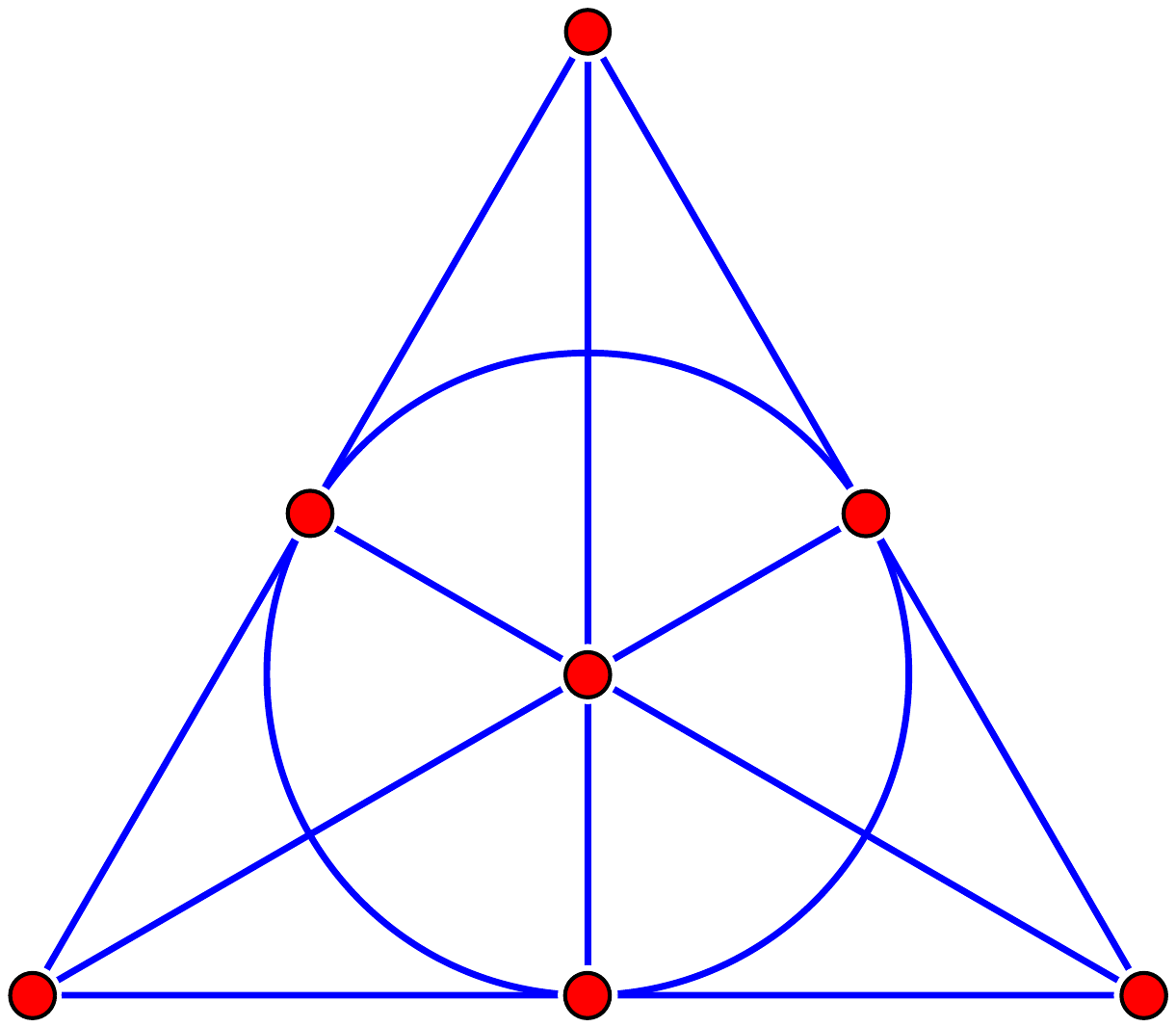}\label{fig: fano}} \hspace{2cm}
    \subfloat[][non-Pappus matroid]{\includegraphics[height=0.24\textwidth]{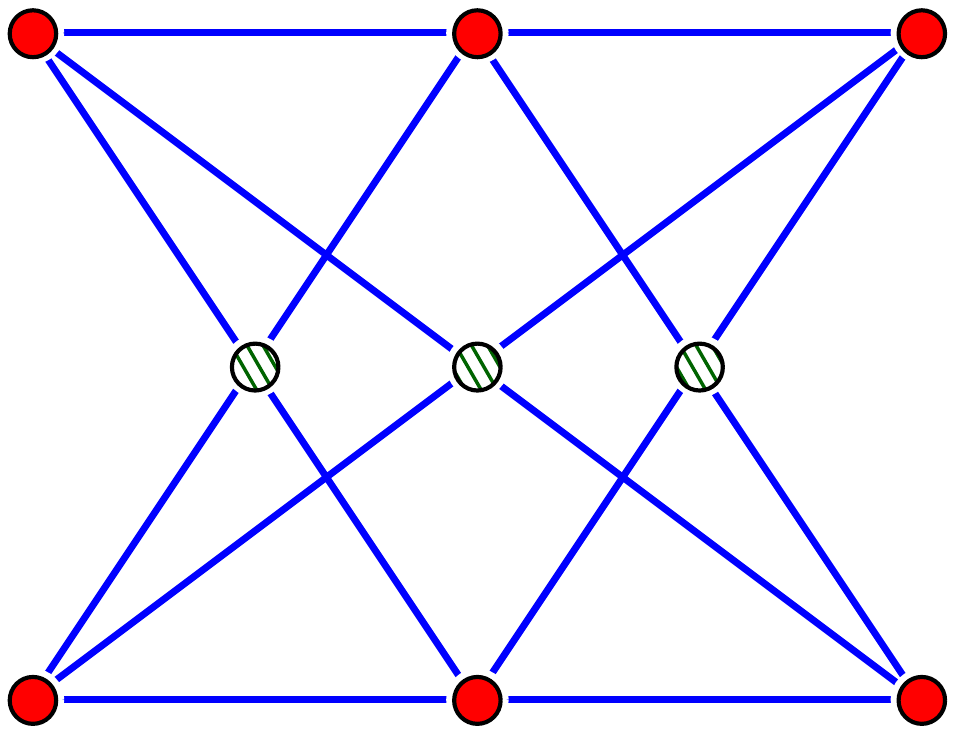}\label{fig: nonpappus}}
\caption{The Fano and non-Pappus matroids. These are rank $3$ matroids visualized as
follows: the elements of the ground set are the points; every set of three points is
independent unless there
 is a curve going through it; no set of four points is
independent. The non-Pappus matroid gets its name from the fact that Pappus's theorem
in projective geometry implies that the hatched points in (B) must be collinear, but
they are independent in the non-Pappus matroid.}
\end{figure}
The next year, Mac Lane published a paper \cite{maclane} attributing to Whitney an
example of a rank $3$ matroid on the set $\{1,\ldots,9\}$ whose dependencies are given
in Figure \ref{fig: nonpappus}. This matroid has become known as the non-Pappus
matroid, because (as Mac Lane notes) it forces a violation of Pappus's theorem.
Pappus's theorem is valid over all fields, so Mac Lane's example is the first
published matroid not representable over any field.

Whitney introduced what he called the ``cycle matroid of a graph'' \cite{whitney}
which has come to be called a {\em graphic matroid}. Given a graph $G = (V,E)$ we
define the set $\cI_G$ to be the subsets of edges that do not contain any circuits. At
the heart of the verification that these sets satisfy the axioms in Definition
\ref{def: independent} is the fact that all maximal independent sets in a connected
component of a graph are spanning trees. 
\begin{ex}
Consider the complete bipartite graph $K_{2,3}$ in Figure \ref{fig:K23} and define the
matroid $(\{a,b,c,d,e,f\},\cI_{K_{2,3}})$. We depict a basis $\{a,b,e,f\}$ in Figure
\ref{fig: K23_Ind} and a circuit $\{a,b,d,e\}$ in Figure \ref{fig: K23_circuit}. The
reader may notice that we again have a set of size six (edges, in this case) whose
elements satsify the same dependence relations as in Examples \ref{ex: matrix} and
\ref{ex: vec config}.

	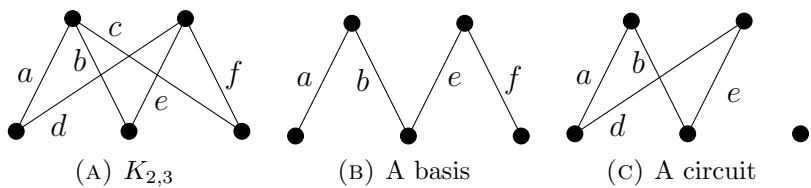
\begin{figure}[htbp]
	\centering
	\subfloat[][$K_{2,3}$]{
			\begin{tikzpicture}[scale=1.5]
			\node (1) at (0,0) [left]{};
			\node (2) at (-1,0)[left] {};
			\node (3) at (1,0) [left]{};
			\node (4) at (-.5,1) [left]{};
			\node (5) at (.5,1) [left]{};
			\filldraw (0,0) circle (2pt);
			\filldraw (-1,0) circle (2pt);
			\filldraw (1,0) circle (2pt);
			\filldraw (-.5,1) circle (2pt);
			\filldraw (.5,1) circle (2pt);
			\draw (-1,0) -- (-.5,1) node [midway, left]{$a$} -- (0,0) node [very near start, below=2pt] {$b$}-- (.5,1) node [near start, right] {$e$}-- (1,0)  node [midway, right]{$f$};
			\draw (-.5,1) -- (1,0) node [near start, above]{$c$};
			\draw (.5,1) -- (-1,0) node [near end, below]{$d$};
			\end{tikzpicture}
			\label{fig:K23}}
    \subfloat[][A basis]{
    \begin{tikzpicture}[scale=1.5]
			\node (1) at (0,0) [left]{};
			\node (2) at (-1,0)[left] {};
			\node (3) at (1,0) [left]{};
			\node (4) at (-.5,1) [left]{};
			\node (5) at (.5,1) [left]{};
			\filldraw (0,0) circle (2pt);
			\filldraw (-1,0) circle (2pt);
			\filldraw (1,0) circle (2pt);
			\filldraw (-.5,1) circle (2pt);
			\filldraw (.5,1) circle (2pt);
		\draw (-1,0) -- (-.5,1) node [midway, left]{$a$} -- (0,0) node [midway, left] {$b$}-- (.5,1) node [midway, right] {$e$}-- (1,0)  node [midway, right]{$f$};
			\end{tikzpicture}
    \label{fig: K23_Ind}}
    \subfloat[][A circuit]{
    \begin{tikzpicture}[scale=1.5]
			\node (1) at (0,0) [left]{};
			\node (2) at (-1,0)[left] {};
			\node (3) at (1,0) [left]{};
			\node (4) at (-.5,1) [left]{};
			\node (5) at (.5,1) [left]{};
			\filldraw (0,0) circle (2pt);
			\filldraw (-1,0) circle (2pt);
			\filldraw (1,0) circle (2pt);
			\filldraw (-.5,1) circle (2pt);
			\filldraw (.5,1) circle (2pt);
			\draw (-1,0) -- (-.5,1)node [midway, left]{$a$}  -- (0,0)node [very near start, below=2pt] {$b$} -- (.5,1) node [midway, below right] {$e$}-- (-1,0)node [very near end, below=2pt, right=1pt]{$d$};
			\end{tikzpicture}
    \label{fig: K23_circuit}}
\caption{A basis and a circuit for the graphic matroid on $K_{2,3}.$}
\end{figure}
\end{ex}

Kung \cite[page 18]{kung} notes that a ``curious feature of matroid theory, not
shared by other areas of mathematics is that there are many natural and quite
different ways of defining a matroid.'' Rota expresses a similar sentiment in his
introduction to \cite{kung}:

\begin{quotation} 
\ldots the unique peculiarity of this field, the exceptional variety of cryptomorphic
definitions for a matroid, embarassingly unrelated to each other and exhibiting wholly
different mathematical pedigrees. \end{quotation}Indeed, the axioms defining a matroid
can be reformulated in terms of bases, rank, dependent sets, or circuits.  A number 
of reference works (e.g.,  \cite{welsh, oxley, kung}) describe all of these
in detail.  Since we need them in what follows, we now state the axiomitization of matroids by 
circuits.

\pagebreak

\begin{defin}\label{def: dependent}
	A \emph{matroid} is a pair $(E,\cC)$, 
	where $E$ is a finite set and $\cC \subseteq 2^E$ satisfies
	\begin{enumerate}
		\item $\emptyset\notin \cC$.
		\item If $C_1\in \cC$ and $C_2\subsetneq C_1$, then $C_2\notin \cC$.
		\item If $C_1, C_2 \in \cC$, then for any $x\in C_1\cap C_2$, there is 
		a $C_3\in \cC$ such that 
		$C_3\subseteq (C_1\cup C_2)\setminus \{x\}$.
	\end{enumerate}
	The sets in $\cC$ are the \emph{circuits} of the matroid.
\end{defin}
Here, too, the first two axioms are more intuitive than the third. The third axiom,
known as the   ``circuit elimination axiom,'' 
is natural from the point of view of linear algebra, as
two dependence relations in which a vector $\bx$ appears with nonzero coefficient can
be combined to get a new dependence relation in which $\bx$ has been eliminated.

\section{Matroids via elimination and projection.} \label{sec: elim and project}	
The first definition of an algebraic matroid that we will present is formulated in terms of
a prime ideal in a polynomial ring. Circuits will be encoded via certain \emph{circuit
polynomials}.  To verify that our definition indeed gives a matroid, we establish the 
circuit elimination axiom using classical elimination theory.
Results in the area may be attributed to B\'ezout in the 18th century and later to Cayley,
Sylvester, and Macaulay in the 19th century and early 20th century. Elimination theory
fell out of fashion in the mid-twentieth century; Weil \cite{weil} wrote that work of
Chevalley on extensions of specializations ``eliminate[s] from algebraic geometry the
last traces of elimination-theory...,'' illustrating the attitude of that era.
However, computational advances in the last 40 years ignited a resurgence of interest
in elimination theory, famously inspiring Abhyankar \cite{Abhyankar2004} to write a poem
containing the line ``Eliminate the eliminators of elimination theory.'' We briefly
review the relevant results from elimination theory  and then define algebraic matroids.

\subsection{Elimination theory and resultants}\label{sec: elim}
We will typically be working with a polynomial ring $k[x_1,\ldots, x_r]$, and our goal
will be to eliminate a single variable, say $x_r,$ from two irreducible polynomials
$p(x_1, \ldots, x_r)$ and $q(x_1, \ldots, x_r)$ by finding polynomials $A(x_1, \ldots,
x_r)$ and $B(x_1, \ldots, x_r)$ so that $Ap+Bq$ is a polynomial in $k[x_1, \ldots,
x_{r-1}]$. For example, we might want to eliminate the variable $d$ in the polynomials
$p = ae-bd$ and $q = af-cd$ in Example \ref{ex: matrix}. We see that $cp-bq = ace-bcd
- abf +bcd = ace-abf$ is a polynomial combination of $p$ and $q$ not containing $d.$
Now we explain how this kind of elimination can be performed in general.
	 
Let $R$ be an
integral domain (typically $R = k[x_1, \ldots, x_{r-1}]$) and $R[x]$ be
the ring of polynomials in $x$ with coefficients in $R$. We denote by $R[x]_{< n}$ the
$R$-submodule of polynomials of degree less than $n$ in $x$. With this notation we can
define the resultant.
\begin{defin}[Sylvester's resultant]\label{def: resultant}
	Let $R$ be an integral domain and let $p$ and $q$
	be polynomials of degrees $m$ and $n$ in $R[x]$.  The 
	map $(a,b)\mapsto ap + bq$ is an $R$-linear map
	$R[x]_{< n}\oplus R[x]_{< m}\to R[x]_{< n+m}$.  The \emph{resultant}
	$\Res(p,q,x)$ is the determinant of this map.
\end{defin}

For example, we can perform the previous elimination of $d$ from $ae-bd$ and $af-cd$
by taking the determinant of
	\[
	\begin{pmatrix}
	-b & -c\\
	ae & af
	\end{pmatrix}.
	\]
Theorem \ref{thm: resultant} tells us that if $p$ and $q$ have no common factors, then
$\Res(p,q,x)$ is a polynomial combination of $p$ and $q$ in which $x$ has been
eliminated. An account of the proof can be found in  \cite[Section 3.6]{CLO}.
% where
% the reader can also find the matrix of the map in Definition \ref{def: resultant}
% given explicitly in coordinates.
	
\begin{thm}\label{thm:  resultant}
	The resultant of polynomials $p$ and $q$ in $R[x]$ satisfies
	the properties:
	\begin{enumerate}
		\item $\Res(p,q,x)\in \langle p,q\rangle \cap R$;
		\item $\Res(p,q,x) \equiv 0$ if and only if $p$ and $q$
		have a common factor in $R[x]$ of \\ positive 
		degree in $x$.
	\end{enumerate}
\end{thm}
We will apply this theorem to distinct irreducible polynomials in a 
prime ideal. Since we need some flexibility in terms of which 
variable to eliminate, we define the \textit{support} of a 
polynomial $p\in k[x_1,\ldots, x_n]$ to be the set of variables appearing in it. This next corollary summarizes what we need.
\begin{cor}\label{cor: resultant}
	Let $k$ be a field, and $P$ be an ideal in 
	$k[x_1,\ldots, x_n]$.  If $p$ and $q$ are different 
	irreducible polynomials in $P$ both supported on $x_n$, then 
	$0\neq \Res(p,q,x_n)\in P\cap k[x_1,\ldots, x_{n-1}]$.
\end{cor}
\begin{proof}
	Since $p$ and $q$ are irreducible, they don't have a 
	common factor.  Since $p$ and $q$ are in $P$, certainly
	$\langle p,q\rangle\subseteq P$.  Theorem \ref{thm:  resultant}
	tells us that 
    $0\neq \Res(p,q,x_n)\in \langle p,q\rangle\cap k[x_1,\ldots, x_{n-1}]
	\subseteq P\cap k[x_1,\ldots, x_{n-1}]$.
\end{proof}

\subsection{Algebraic matroids from prime ideals}\label{sec: coord}
    
Given a set of polynomial equations, we can ask what dependencies they
introduce on the variables. If our set of polynomials is 
a prime ideal,  these dependencies satisfy the matroid axioms. 
The characterization of independent coordinates modulo an
ideal in Definition \ref{def: coordinate matroid} can be deduced from 
the definition of independence for elements in a field extension, 
which we give in the next section. We will go in the other direction, 
giving an elementary proof that seems to be folklore.
	
Let $k$ be a field, $E = \{x_1,\ldots, x_n\}$ be
a set of  variables. 
For any $S \subseteq E$ we define
$k[S]$ to be the set of polynomials with variables in $S$ 
and coefficients in $k$.  
	
\begin{defin}\label{def: coordinate matroid}
	Let $k$ be a field, $E = \{x_1, \ldots, x_n\}$ 
	and $P$ b a prime ideal
	in $k[E]$.  Given $S \subseteq E$, we define 
 		\[
 		\cI_P = \{S \subseteq E \mid P \cap k[S]  = \langle 0\rangle\}
 		\]
	to be the set of all subsets of $E$ that are independent modulo $P$. 
    The dependent sets $\cD_P$ are the subsets of $E$ not in $\cI_P$.
 	\end{defin}  
This notion of independence depends on our choice of coordinates. For example, if $P =
\langle x,y \rangle\subseteq k[x,y,z]$, then $\cI_P$ contains a single maximal
independent set, $\{z\}$. However, the ideal $Q = \langle x+2y+3z, x+5y+2z\rangle$,
which can be obtained from $P$ via a linear change of coordinates has three maximal
independent sets, $\{x\}, \{y\},$ and $\{z\}$.  
It is also the case that very different ideals can give
rise to the same independent sets. For example, if $T = \langle x^2-y, xy-z \rangle$,
then the maximal independent sets of $\mI_T$ are $\{x\}, \{y\},$ and $\{z\},$ which
are the same as in those in $\cI_Q$.
	 
We now show that the elements of $\cI_P$ are the independent sets of a matroid. First
we will show that every minimal dependent set $C$ is encoded by an irreducible
polynomial $f_C$ that is unique up to scalar multiple.
	
\begin{thm}\label{thm: circuitpoly}
	Let $k$ be a field and $P$ a prime ideal 
	in $k[x_1,\ldots, x_n]$.  Let $C\subseteq E$.
	If $P\cap k[C]\neq \langle 0 \rangle$ and $P\cap k[C'] = \langle 0\rangle$
	for all $C'\subsetneq C$, then $P\cap k[C]$ is principal and 
	generated by an irreducible polynomial $f_C$.  The support of 
	$f_C$ is all of $C$.
\end{thm}
\begin{proof}
	First suppose that $f\in P\cap k[C]$ is a nonzero polynomial.  Since 
	$k[S]$ is a unique factorization domain, 
	$f$ is a product of irreducible factors $f_1\cdots f_k$.
	Because $P$ is prime, at least one of the $f_i$ is in $P$.  
	Thus, every $f\in P\cap k[C]$ has an irreducible
	factor in $P$.
		
	By the minimality hypothesis on $C$, any polynomial $g$ in 
	$P\cap k[C]$ is supported on all of $C$.  In particular, 
	if $g$ and $h$ are both in $P\cap k[C]$, they must 
	be supported on a common variable, $x_i \in C$.
	When $g$ and $h$ are irreducible, we then have
	the  situation of Corollary \ref{cor: resultant}.
	If $g\neq h$, this implies that 
	$P\cap k[C\setminus \{x_i\}] \neq \langle 0\rangle$.
		
	Using the line of reasoning above, if $f_i$ and $f_j$ are distinct irreducible
factors of $f$ in $P \cap k[C]$, then we can eliminate a variable in common to them,
contradicting the minimality of $C.$ Therefore, we conclude that $f$ is divisible by a
unique irreducible factor in $P \cap k[C],$ which we denote by $f_C.$ Again, by the
minimality of $C$, we can see that $f_C$ must be the unique irreducible polynomial in
$P \cap k[C],$ and that it divides every polynomial in $P \cap k[C].$
	\end{proof}
    
The polynomial $f_C$ appearing in the conclusion of Theorem \ref{thm: circuitpoly} is
called the \emph{circuit polynomial} of the circuit $C$ in $(E,\cI_P)$. This notion
first appears in the paper of Dress and Lov\'asz \cite{dressLovasz}. It was later
explored, in a statistical context, by Kir\'aly and Theran \cite{kiralyTheran2}. The
unpublished preprint of Kir\'aly, Rosen, and Theran \cite{KRT}, 
where this use of 
the term ``circuit polynomial'' originates, 
studies how symmetries of an algebraic matroid are reflected 
in the associated circuit polynomials.
		
Now we are ready show that the sets in
Definition \ref{def: coordinate matroid} are the independent sets of a matroid. 
Instead of checking the independent set axioms directly, we use the circuit 
axioms.
        
\begin{thm}\label{thm: coordinate matroid}
    The pair $(E,\cI_P)$ from Definition \ref{def: coordinate matroid} is 
    a matroid.
\end{thm}
\begin{proof}
With respect to $\cI_P$, the dependent subsets  
are those sets $S$ for which $P\cap k[S]\neq
\langle 0\rangle$. We define $\cC_P$ to be the dependent subsets of $E$ that are
minimal with respect to inclusion. The result will follow once we have checked the
circuit axioms from Definition \ref{def: dependent}.

Certainly $P\cap k[\emptyset]$ is the zero ideal, which implies that $\emptyset
\notin \cC_P$, which is axiom (1). Minimality of the circuits gives axiom (2) by
definition. For use later, we note that if $P\cap k[D]\neq \langle 0\rangle$ then some
subset of $D$ is a circuit by axiom (1).
	    
The interesting axiom is (3). Suppose that $C_1$ and $C_2$ are circuits in $\cC_P$
with $x_i\in C_1\cap C_2$. By Theorem \ref{thm: circuitpoly}, there are distinct
irreducible polynomials $f_{C_1}$ and $f_{C_2}$ both supported on 
$x_i$ in $P\cap k[C_1\cup C_2]$. Corollary \ref{cor: resultant} then implies 
that there exists a nonzero $h \in P\cap k[(C_1\cup C_2)\setminus \{x_i\}]$. 
Since the support of $h$ is contained in 
$(C_1\cup C_2)\setminus \{x_i\}$, $P\cap k[(C_1\cup C_2)\setminus
\{x_i\}] \neq \langle 0\rangle$. 
By the observation above, we have axiom (3).
\end{proof}

Every linear matroid is algebraic, though not all algebraic matroids are linear. To
see that a linear matroid is algebraic, suppose we are given a matroid on the columns
of a $d \times n$ matrix $M$. Let $B = (\bb_1\, \cdots\, \bb_r)$ be a matrix whose columns
form a basis for the kernel of $M$. We'll use the vectors $\bb_i$ to define an ideal
generated by linear forms in the ring $k[x_1,\ldots, x_n].$ Define linear forms $L_1,
\ldots, L_r$ by setting $L_i = \bb_i \cdot (x_1,\ldots,x_n).$ An ideal generated by 
linear forms must
be prime, so $P = \langle L_1, \ldots, L_r \rangle$ defines a matroid $\cI_P,$ where
the linear forms defining dependent sets of variables exactly record the dependencies
among the columns of $M.$

\subsection{Varieties and projections}\label{sec: varieties}
We will construct a geometric counterpart to coordinate matroids, using 
projections of varieties, which we briefly introduce.

Let $k$ be a field and $f_1,\ldots, f_m$ be polynomials in $k[x_1,\ldots, x_n].$ The
common vanishing locus of these polynomials is the \emph{algebraic set} 
\[V =
V(f_1,\ldots, f_m) = \{ \bp \in k^n \mid f_1(\bp) = \cdots = f_m(\bp) = 0\}.
\]
In the \emph{Zariski topology} on $k^n$ a 
set is closed if and only if it is an algebraic set.  The 
\textit{Zariski closure}, denoted $\overline{X}$, 
of a subset $X\subseteq k^n$ is the smallest 
algebraic set containing $X$.  An
algebraic set is called \emph{irreducible} if it is not union of two nonempty
algebraic sets, and we call an irreducible algebraic set a \emph{variety}. For
example, the three equations 
%$ae-bd = af-cd = bf-ce=0$
$ae-bd =0, af-cd =0,$ and $bf-ce=0$ 
in Example \ref{ex:
matrix} define a variety in $k^6$ whose points correspond to $2\times 3$ matrices with
rank at most one.

Although we have defined an algebraic set as the solution set of a finite system of
polynomial equations, it is not hard to check that if $I = \langle f_1,\ldots, f_m
\rangle$ is the ideal generated by the polynomials $f_i$, then $V(f_1,\ldots, f_m) =
V(I).$ Conversely, given an algebraic set $V \subseteq k^n,$ one may define 
\[
    I(V) = \{ f \in k[x_1,\ldots, x_n] \mid f(\bp) = 0 \ \forall \bp \in V\},
\]
the ideal of all polynomials that vanish on $V.$

What is the relationship between an algebraic set and its vanishing ideal?
If $V = V(f_1, \ldots, f_m)\subseteq k^n$ is algebraic, then $V(I(V)) = V$.  
Since $\{f_1, \ldots f_n\}\subseteq I(V)$, $V(I(V))\subseteq V$.  By 
definition, every $f\in I(V)$ vanishes on $V$, so $V\subseteq V(I(V))$.
Starting from an ideal $I\subseteq k[x_1, \ldots, x_n]$, we don't 
necessarily have $I(V(I)) = I$. For example, if 
$k = \RR$ and $I = \langle x^2 + y^2 +1 \rangle,$ then 
$V(I) = \emptyset,$ so $I(V(I)) = \langle 0 \rangle.$ 
However, over an algebraically closed
field, Hilbert's famous Nullstellensatz says that $I(V(I)) = I$ holds 
if $I$ is a radical ideal.  In this setting, the fundamental 
``algebra-geometry dictionary'' (see, e.g., \cite[chapter 4]{CLO})
says that there is a bijection $V\mapsto I(V)$ between irreducible varieties in $k^n$ and  
prime ideals in $k[x_1, \ldots, x_n]$.

Now that we have a geometric counterpart to prime ideals when $k$ is 
closed, we need an analogue for elimination.  For $S\subseteq \{1, \ldots, n\}$
define a projection  $\pi_S:k^n \to k^{|S|}$
by $\pi_S(p_1,\ldots, p_n) = (p_i \mid i \in S),$ 
where we preserve the order of the coordinates.  
We are 
going to be comparing two kinds of objects, so we take 
$E = \{1,\ldots, n\}$ as the common ground set that indexes both 
variables and standard basis vectors of $k^n$.  For 
$S\subseteq E$, extend the notation $k[S]$ to means $k[x_i : i\in S]$.

If $V\subseteq k^n$ is algebraic, then $\pi_S(V)$ corresponds to eliminating 
the variables not in $S$ from its vanishing ideal 
$I\subseteq k[x_1, \ldots,x_n]$.  Suppose that 
$f\in I\cap k[S]$ and $\bp\in V$.  Certainly $f(\bp) = 0$;
more interestingly, since $f$ only sees variables in $S$, $f(\pi_S(\bp)) = 0$
as well.  Hence $\pi_S(V)\subseteq V(I\cap k[S])$; since $V(I\cap k[S])$
is closed, it contains $\overline{\pi_S(V)}$ as well.  When $k$ is 
closed and $V$ is irreducible, 
we can get more.  This affine version of the ``closure theorem'' 
is a key technical tool for us.
\begin{thm}\label{thm: closure}
	Let $k$ be an algebraically closed field, and let $V\subseteq k^n$
	be an irreducible algebraic set with ideal 
	$I = I(V) \subseteq k[x_1,\ldots, x_n]$.  Then for all $S\subseteq E$:
	\begin{enumerate}
		\item $\overline{\pi_S(V)}$ is irreducible; 
		\item $I(\overline{\pi_S(V)}) = I\cap k[S]$.
	\end{enumerate}
\end{thm}
\begin{proof}
Since $V$ is irreducible, $I$ is prime, which implies that $I \cap k[S]$ is also
prime. By  \cite[Theorem 3.2.3]{CLO}, $V(I \cap k[S]) =\overline{\pi_S(V)}$, which
shows that $\overline{\pi_S(V)}$ is irreducible, and (2) follows by the
Nullstellensatz.
\end{proof}
To make this theorem work, taking the Zariski closure of the image was essential. 
For instance, we noted that the set $S =\{a,b,e,f\}$ is independent in Example \ref{ex: matrix}.
However, if $V = V(ae-bd, af-cd, bf-ce)$, the projection $\pi_S: V \to k^4$ cannot be
surjective because a point with $a=1, b=0, e=1, f=1$ cannot come from a rank one
matrix because if $b=0$ the equation $ae-bd=0$ implies that either $a$ or $e$ is zero.

%\subsection{Algebraic matroids via projections} \label{sec: project}
Now we define a geometric analogue of coordinate matroids.
	
\begin{defin}\label{def: basis matroid}
	Let $k$ be an algebraically closed field and let 
	$V\subseteq k^n$ be an irreducible variety.
	Define 
	\[
	\cI_V = \{ S\subseteq E \mid \overline{\pi_S(V)} = k^{|S|}\}.
	\]
\end{defin}
To check that we have defined a matroid, instead of 
verifying the axioms, we will use the relationship 
between projection and elimination to relate $\cI_V$ to 
a coordinate matroid.
\begin{thm}\label{thm: basis matroid}
The set $\cI_V$ from Definition \ref{def: basis matroid} gives the independent sets of
a matroid on $E$. We call this the \emph{basis projection matroid}.
\end{thm}
\begin{proof}    
Let $P$ be the vanishing ideal of $V$. Since $k$ is closed, the algebra-geometry dictionary tells us that 
$P$ is prime. Hence the coordinate
matroid $(\{1,\ldots,n\},\cI_P)$ is defined.  The set $\cI_V$ is the
same as $\cI_P$, since,
		\[
		S\in \cI_V \Longleftrightarrow \overline{\pi_S(V)} = k^{|S|}
		\overset{\text{Thm. \ref{thm: closure}}}{\Longleftrightarrow}
		P\cap k[S] = \langle 0 \rangle \Longleftrightarrow
		S\in \cI_P.
		\]
Hence $(E,\cI_V)$ is a matroid.
\end{proof}

The advantage of basis projection matroids is that sometimes it is more convenient to think
geometrically. In Example \ref{ex: rigidity}, the fibers of the projection map contain
useful geometric information. For a fixed $G\subseteq \binom{n}{2}$, if $\ell_G$ is the
vector $(\ell_{ij}: ij\in G)$, then the fiber $\pi_G^{-1}(\ell_G)$ tells us about the
achievable distances between pairs of points outside of $G$, a perspective 
emplyoed by Borcea \cite{borcea}, Borcea and Streinu \cite{borceaStreinu}, 
and Sitharam and Gao \cite{meera}.
Similarly, in Example \ref{ex: matrix}, the fibers of the projection map are the 
``completions'' of a low-rank matrix from the observed entries.

\section{Algebraic matroids and field theory.} \label{sec: field}
Classically, algebraic matroids are defined in terms of
field extensions. Let $k$ be a field and $K \supset k$  a field
extension. We say that 
$S = \{\alpha_1, \ldots, \alpha_n\} \in K$ are
\emph{algebraically dependent} over $k$ if there exists a nonzero polynomial $f \in k[x_1,
\ldots, x_n]$ with $f(\alpha_1, \ldots, \alpha_n) =0.$ If no such polynomial exists we
say the elements are \emph{algebraically independent} over $k$.
\begin{defin}\label{def: fieldMatroid}
Let $K \supset k$ be an extension of fields and $E = \{\alpha_1, \ldots, \alpha_n\}$
be a subset of $K \backslash k.$ Without loss of generality, we assume that $K = k(E)$.  We define a matroid $(E, \cI_K)$ with ground set $E$
and define $S \subseteq E$ to be an independent set if $S$ is algebraically independent
over $k$.  %The rank of $(E,\cI_K)$ is the transcendence degree of $K$ over $k$.
\end{defin}
The classical definition is equivalent to the ones in terms of ideals
and varieties.
If $k \subseteq E = \{\alpha_1, \ldots, \alpha_r\} \subseteq K$, and we define $\varphi:
k[x_1,\ldots, x_r] \to k[\alpha_1, \ldots, \alpha_r]$ by $\varphi(x_i) = \alpha_i,$
then $P = \ker \varphi$.  The independent sets of the coordinate matroid
$(\{x_1,\ldots,x_n\},\cI_{\ker \varphi})$ correspond naturally to independent subsets of 
$(E,\cI_K)$.  Moreover, this construction can be reversed.  If we start with a prime ideal $P$, $k[\{x_1,\ldots, x_n\}]/P$, 
is an integral domain.  Hence its field of fractions $K\supset k$ is defined.
Defining $\{\alpha_1, \ldots, \alpha_n\}$ as the elements of $K$ corresponding 
to the $x_i$ produces an algebraic matroid $(E,\cI_K)$
with independent sets corresponding to those in $(\{x_1,\ldots,x_n\},\cI_P)$.  We saw that there is a natural correspondence between basis projection matroids and coordinate matroids if $k$ is algebraically closed in Theorem \ref{thm: basis matroid}.

We are now in possession of three different-looking, as Rota puts it 
``cryptomorphic'', definitions of an algebraic matroid.  As an illustration 
of why this is useful, we consider the rank, which is an important quantity in almost any 
application.  %From geometric intuition, 
% we would guess that if $(E,\cI_V)$ is a basis projection matroid over an algebraically closed field, the rank should 
% be $r = \dim V$, since we know that $V$ projects dominantly onto $k^r.$  %It is much easier to turn this into a proof using field extensions than geometric arguments.
The rank of an algebraic matroid  $(E, \cI_K)$ is the transcendence degree of $K$ over $k$.  
Via the correspondences above, we can also see that this gives the rank for the associated coordinate and basis projection matroids.  
Does the rank of a coordinate matroid $(E,\cI_P)$  or a basis projection matroid $(E,\cI_V)$ have any meaning?  The 
answer is yes, they are the dimensions of $k[E]/P$ and $V$, respectively.  This is difficult to see 
directly, even with the (somewhat technical, see, \cite[Chapter 8]{eisenbud})
definitions of dimension for ideals and varieties.  However, both quantities are known to be 
equal to the transcendence degree of the extension we constructed to go between $(E,\cI_P)$ and 
$(E,\cI_K)$. 

In the case of field extensions the theory follows from work of van der Waerden who showed that ``[t]he algebraic dependence relation has the following
fundamental properties which are completely analogous to the fundamental properties of
linear dependence,'' \cite[Ch. VIII, S. 64]{vdW}. The
connection was also known to Mac Lane, who wrote about lattices of subfields in
\cite{maclane} and drew attention to ``connection[s] to the matroids of Whitney'' and
the ``lattices by Birkhoff.''
	
It seems that algebraic matroids were largely forgotten after Mac Lane until the work
of Ingleton in the 1970s.  Ingleton asked the basic question for algebraic matroids 
that Whitney had already considered in the 1903's for linear representability: \textit{is 
every matroid realizable as an algebraic matroid?} This was answered by Ingleton and
Main \cite{ingletonMain} in the negative who showed that the V\'{a}mos matroid,
displayed in Figure \ref{fig: vamos}, is not algebraic.

\begin{figure}[h]
    \centering
    \includegraphics[width=0.20\textwidth]{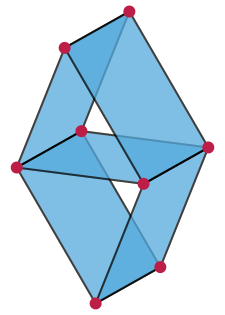}
    \caption{The Vamos matroid.  This is a rank $4$ matroid, shown according
    to the convention that all sets of size at most $4$ are independent, except 
    for the size $4$ sets indicated by shaded quadrilaterals.  
    Picture from \cite{epsVamosPicture}.}\label{fig: vamos}
\end{figure}

What about the relationship between algebraic and linear matroids?
In characteristic zero, the two classes are the same.
\begin{thm}[Ingleton \cite{ingleton}] \label{thm: linear}
If a matroid is realizable as an algebraic matroid over a field $k$ of
characteristic zero, then it is also realizable as a linear matroid over $k$.
\end{thm}
What about fields of positive characteristic? 
Whitney's example of a matroid that is not linearly representable over $\CC$ but is 
over $\mathbb{F}_2$ shows that the characteristic of the underlying field matters.
The characteristic of the field also makes a big difference in determining algebraic representability.
In a series of papers in the 1980s
 Bernt Lindstr\"{o}m \cite{Lindstrom1,Lindstrom2,Lindstrom3,Lindstrom4} demonstrated
that there are infinitely many algebraic matroids representable over \emph{every}
characteristic besides zero, and not linearly representable over 
\emph{any} field.  Characterizing which matroids are algebraic (in positive characteristic) 
is an active area of research, including the recent advances of Bollen, Draisma and
Pendavingh \cite{bollen} (see also Cartwight \cite{cartwright}).

\section{Applications.}\label{sec: applications}
We revisit the earlier examples, including matrix completion, rigidity theory, and
graphical matroids, from the point of view of algebraic matroids, highlighting the
connections revealed by the common language.
    
\subsection{A matrix, an ideal, and a variety}\label{sec: matrix, ideal, variety}
An $m \times n$ matrix $A = (\ba_1\, \cdots\, \ba_n)$ with $\ba_i \in \Z^m$ gives rise to
a matroid that can be realized as a linear matroid and a coordinate matroid in a
natural way via the construction of the \emph{toric variety} $X_A$ associated to $A$.
(Of course, once we have the coordinate matroid we also have the basis matroid and the
algebraic matroid of the field of fractions of the coordinate ring of $X_A.$)
    
From the data of $A$ we get a map $\varphi_A:(\CC^*)^m \to \CC^n$ given by
$\varphi_A(\bt) = (\bt^{\ba_1}, \ldots, \bt^{\ba_n}),$ where $\bt^{\ba}=t_1^{a_1}
\cdots t_m^{a_m}$. As shown in \cite{sturmfels}, the variety $X_A$ defined to be the
Zariski closure of the image of $\varphi_A$ has ideal $I_A = \langle \bx^{\bu} -
\bx^{\bv} \mid \bu, \bv \in \Z_{\geq0}, \bu-\bv \in \ker A \rangle.$ Since elements of
$\ker A$ are dependence relations on the columns of $A$, we see that the linear
matroid on the columns of $A$ is the same as the algebraic matroid defined by the
ideal $I_A.$ The variety $X_A$ is called a toric variety as it contains the torus
$(\CC^*)^m$ as a dense open subset.
    
Returning to Example \ref{ex: vec config}, we see that the columns of 
     \[
      A=  \begin{pmatrix}
        1 & 1 & 1 & 0 & 0 & 0\\
        0 & 0 & 0 & 1 & 1 & 1\\
        1 & 0 & 0 & 1 & 0 & 0\\
        0 & 1 & 0 & 0 & 1 & 0\\
        0 & 0 & 1 & 0 & 0 & 1
        \end{pmatrix}
    \]
define a parameterization $\varphi_A(\bt) = (t_1t_3, t_1t_4, t_1t_5, t_2t_3, t_2t_4,
t_3t_5).$ If we give the target space coordinates $a,\ldots, f,$ then
$(1,0,0,1,0)-(0,1,0,1,0,0) \in \ker A$, and this tells us that the polynomial $ae-bd$ is in
$I_A.$ (Indeed, if we let $\psi = ae-bd,$ then $\psi(\varphi(\bt)) =
(t_1t_3)(t_2t_4)-(t_1t_4)(t_2t_3)=0.$)
    
So, the linear dependence relations on the columns of $A$ give algebraic dependence
relations on $a, \ldots, f.$ This is true for any general toric variety $X_A$ that
arises from an integer matrix $A$ in this way. For more detail on how the circuits of
the matroid on the columns of $A$ are related to the ideal $I_A$, see
\cite[chapter 4]{sturmfels}. We will soon see that  the matrix
in Example \ref{ex: vec config} has a special form that provides 
a connection to the rank one matrix completion problem.
    
\subsection{Matrix completion, varieties, and bipartite graphs} \label{sec: matrix completion}
Algebraic matroids were used to study the matrix completion problem by Kir\'aly, Theran, and Tomioka
\cite{kiralyTheran}. We now provide a brief introduction.
    
Define $I_{ m\times n,r}$ to be the ideal generated by the $(r+1) \times (r+1)$ minors
of the generic matrix $M = (\bx_1\, \cdots\, \bx_n)$ where $\bx_i$ is a column vector of $m$
indeterminates. This ideal is prime, so defines an algebraic matroid, $\cM_{I_{m\times
n,r}} = (\{(1,1),\ldots,(m,n)\},\cI_{I_{m\times n,r}})$. This is the matroid on the
entries of a general $m\times n$ matrix of rank $r.$
    
\begin{thm}
    The rank of $\cM_{I_{m\times n,r}}$ is $r(m+n-r)$. 
\end{thm}
\begin{proof}[Proof sketch]
The dimension of the variety $V_{m\times n,r}$ of $m \times n$ matrices of rank 
at most $r$ is $r(m+n-r)$. One intuition for this, which isn't far from a proof, is that you
can specify the first $r$ rows and columns of the matrix freely and then rest of the
matrix is determined. This process sets $rm + rn -r^2$ entries in total. 
\end{proof} 
   
Why are the elements of $I_{2\times 3,1}$ the same as the polynomials that vanish on
the toric variety $X_A$ discussed above? Observe that the
coordinates of $\varphi_A(\bt) = (t_1t_3, t_1t_4, t_1t_5, t_2t_3, t_2t_4, t_3t_5)$ 
can be rearranged into a matrix:
   \[
   \begin{pmatrix}
   t_1\\
   t_2
   \end{pmatrix}
   \begin{pmatrix}
   t_3 & t_4 & t_5
   \end{pmatrix}
   =
   \begin{pmatrix}
   t_1t_3 & t_1t_4 & t_1t_5\\
  t_2t_3 & t_2t_4 & t_2t_5
   \end{pmatrix}.
   \]
Replacing each product with a distinct variable, we have the matrix of
indeterminates from Example \ref{ex: matrix}:
   \[
   M = \begin{pmatrix}
   a & b & c\\
   d & e & f
   \end{pmatrix}.
   \]
The $2 \times 2$ minors of $M$ are polynomials that vanish on the multiplication 
table by commutativity and associativity:
   \[
   ae-bd = (t_1t_3)(t_2t_4)-(t_1t_4)(t_2t_3) = t_1t_2t_3t_4-t_1t_2t_3t_4=0.
   \] 
   
More generally, any $m \times n$ matrix with distinct variables as entries can be
interpreted as the formal multiplication table of sets of size $m$ and $n$,
respectively. The $2 \times 2$ minors will vanish on the variety parameterized by
these products, the classical \emph{Segre variety} $\mathbb{P}^{m-1} \times
\mathbb{P}^{n-1}.$
   
The combinatorics of the circuits in Example \ref{ex: matrix} can also be encoded in
the bipartite graph $K_{2,3}$ with vertices labeled $t_1,\ldots, t_5$ so that each
edge corresponds to a product $t_it_j$, as shown in Figure \ref{fig: tijgraph}.
Each 4-cycle in this graph corresponds to a $2 \times 2$ minor, and these are exactly
the circuits of the matroid.
The maximal independent sets are 
\[\begin{array}{llllll}
    \{a,b,c,d\}, & \{a,b,c,e\}, & \{a,b,c,f\}, &
    \{a,d,e,f\}, & \{b,d,e,f\}, & \{c,d,e,f\}, \\
    \{a,b,e,f\}, & \{a,c,f,e\}, & \{a,b,d,f\}, &
    \{b,c,d,f\}, & \{a,c,d,e\}, & \text{ and } \{b,c,d,e\}.
    \end{array}
\]
Given (generic) values for the entries in any of these sets there is a unique matrix completion,
because the circuit polynomials are all linear in the missing entry.  
    \begin{figure}[h!]
			\[
			\begin{tikzpicture}[scale=1.5]
			\node (1) at (0,0) [left]{$t_4$};
			\node (2) at (-1,0)[left] {$t_3$};
			\node (3) at (1,0) [left]{$t_5$};
			\node (4) at (-.5,1) [left]{$t_1$};
			\node (5) at (.5,1) [right]{$t_2$};
			\filldraw (0,0) circle (2pt);
			\filldraw (-1,0) circle (2pt);
			\filldraw (1,0) circle (2pt);
			\filldraw (-.5,1) circle (2pt);
			\filldraw (.5,1) circle (2pt);
			\draw (-1,0) -- (-.5,1) -- (0,0) -- (.5,1) -- (1,0) ;
			\draw (-.5,1) -- (1,0);
			\draw (.5,1) -- (-1,0);
			\end{tikzpicture}
			\]
			\caption{\label{fig: tijgraph} $K_{2,3}$.}
    \end{figure}

\subsection{Distance geometry and rigidity theory}
Given $n$ points $\bx_1, \ldots, \bx_n \in \RR^d,$ there are $\binom{n}{2}$ equations
$(\bx_i-\bx_j)\cdot (\bx_i-\bx_j) = \ell_{ij}$ giving the squared distances between
pairs of points. The (closure of) the image of the  \textit{squared length map}
$(\bx_1,\ldots, \bx_n) \mapsto ((\bx_i-\bx_j)\cdot (\bx_i-\bx_j))$ is a variety
$\CM_{d,n}$ in $\RR^{\binom{n}{2}}$ with defining ideal $I_{d,n}$ given by the $(d+1)
\times (d+1)$ minors of the $(n-1)\times (n-1)$ Gram matrix $M_{d,n}$ with $ij$ entry
equal to
	\[
	\begin{cases}
	 2\ell_{in} &  \text{if $i = j$}\\
	 \ell_{in}+\ell_{jn}-\ell_{ij} & \text{if $i \neq j$}
	\end{cases}.
  \]

It follows from work of Whiteley \cite{coning} and Saliola and Whiteley
\cite{saliolaWhiteley}, that $I_{n,d}$ has a matroid isomorphic to the one associated
with the ideal of the $(d+2)\times (d+2)$ minors of a generic symmetric $n\times n$
matrix, modulo its diagonal. This was independently rediscovered by
Gross and Sullivant \cite{grossSullivant}.
    
It is interesting to ask which interpoint distances are needed in order to determine
the rest, generically. This is related to the central question in the theory of the rigidity of bar
and joint frameworks. To formalize this, we may fix a graph $G$ on $n$ vertices and
think of the edges as fixed-length bars and the vertices as universal joints. A
realization of $G$ in $\RR^d$ is a bar-and-joint framework. A graph $G$ has a flexible
realization if the fiber of $\pi_G: \CM_{d,n} \to \RR^{|G|}$ has positive dimension.
    
If $d = 2$, i.e. we are examining bar-and-joint frameworks in the plane, then the rank
of the matroid is $2n-3.$ When $n=4,$ the rigidity matroid is the uniform matroid of
rank 5 on 6 elements as the deletion of any edge of $K_4$ gives a basis. Thus,
quadrilateral, or 4-bar framework, on these joints is a flexible bar-and-joint
framework. The edges form an independent set but not a maximal independent set. Hence,
there are infinitely many possibilities for $\ell_{24}$ in Figure \ref{fig: 4bar}.
However, a braced quadrilateral is a basis of the rigidity matroid. This implies that
the framework is rigid; indeed, there are only two possibilities for $\ell_{24}$ in
Figure \ref{fig: 4barBrace}.
  
	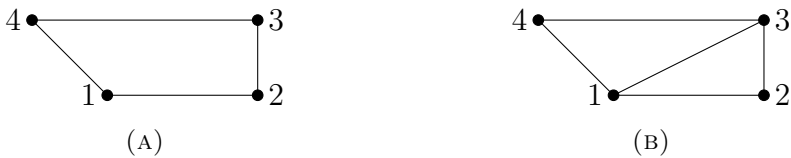
\begin{figure}[ht]
	\centering
		\subfloat[][]{
	\begin{tikzpicture}
	\node (1) at (0,0) [left]{1};
	\node (2) at (2,0)[right] {2};
	\node (3) at (2,1) [right]{3};
	\node (4) at (-1,1) [left]{4};
	\filldraw (0,0) circle (2pt);
	\filldraw (2,0) circle (2pt);
	\filldraw (2,1) circle (2pt);
	\filldraw (-1,1) circle (2pt);
	\draw (0,0) -- (2,0) -- (2,1) -- (-1,1) -- (0,0);
 	\end{tikzpicture}
	\label{fig: 4bar}}
	\qquad\qquad\qquad
    \subfloat[][]{
    \begin{tikzpicture}
	\node (1) at (0,0) [left]{1};
	\node (2) at (2,0)[right] {2};
	\node (3) at (2,1) [right]{3};
	\node (4) at (-1,1) [left]{4};
	\filldraw (0,0) circle (2pt);
	\filldraw (2,0) circle (2pt);
	\filldraw (2,1) circle (2pt);
	\filldraw (-1,1) circle (2pt);
	\draw (0,0) -- (2,0) -- (2,1) -- (-1,1) -- (0,0);
	\draw (0,0) -- (2,1);
 	\end{tikzpicture}
    \label{fig: 4barBrace}}	
    \caption{Examples of frameworks: (A) $4$-bar framework; (B) braced 
    $4$-bar framework.}
\end{figure}
The rigidity matroid has a unique circuit in this case, given by the determinant of 
\[
M_4 = \begin{pmatrix}
2\ell_{14} & \ell_{14}+\ell_{24}-\ell_{12} & \ell_{14}+\ell_{34}-\ell_{13}\\
\ell_{14}+\ell_{24}-\ell_{12} & 2\ell_{24} & \ell_{24}+\ell_{34}-\ell_{23}\\
 \ell_{14}+\ell_{34}-\ell_{13} & \ell_{24}+\ell_{34}-\ell_{23} & 2\ell_{34}
\end{pmatrix},
\]
which has degree two in each variable. This implies that there are two possible
realizations (over $\CC$, counting with multiplicity) for any choice of valid edge
lengths for a basis graph.
    
When $n = 5$ we have a matroid of rank 7 on 10 elements.  There are three bases 
(up to relabeling) corresponding to the graphs in Figure \ref{fig: bases25}.
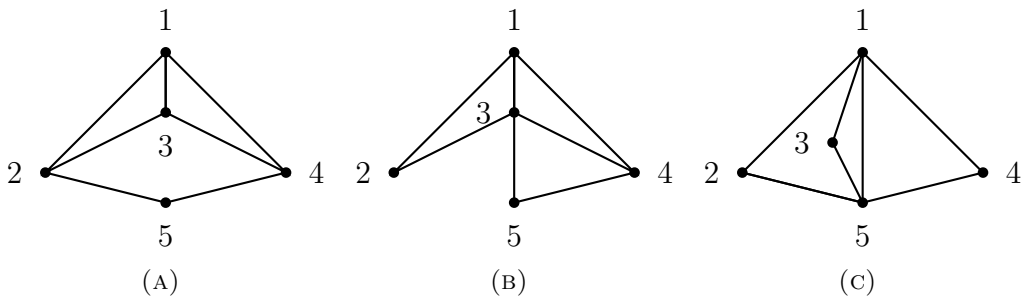
\begin{figure}[h]
  \centering
\subfloat[][]{
\begin{tikzpicture}[style=thick, scale=.8]
\node (1) at (0, 2) [label=above:1]{};
\node (2) at (-2,0) [label=left:2]{};
\node (3) at (0, 1) [label=below:3]{};
\node (4) at (2,0) [label=right:4]{};
\node (5) at (0, -.5) [label=below:5]{};

\filldraw [black] (1) circle (2pt) 
(2) circle (2pt) 
(3) circle (2pt)
(4) circle (2pt)
(5) circle (2pt);

\draw (0,2) -- (0,1) -- (-2,0) --cycle;
\draw (0,2) -- (0,1) -- (2,0) --cycle;
\draw (-2,0) -- (0,-.5) -- (2,0);

%\draw  (5.5, 1) -- (5.5, -1) node[label=below:$L_2$]{};
\end{tikzpicture}}
\subfloat[][]{
\begin{tikzpicture}[style=thick, scale=.8]
\node (1) at (0, 2) [label=above:1]{};
\node (2) at (-2,0) [label=left:2]{};
\node (3) at (0, 1) [label=left:3]{};
\node (4) at (2,0) [label=right:4]{};
\node (5) at (0, -.5) [label=below:5]{};

\filldraw [black] (1) circle (2pt) 
(2) circle (2pt) 
(3) circle (2pt)
(4) circle (2pt)
(5) circle (2pt);

\draw (0,2) -- (0,1) -- (-2,0) --cycle;
\draw (0,2) -- (0,1) -- (2,0) --cycle;
\draw (0,-.5) -- (2,0);
\draw (0,1) -- (0,-.5);

%\draw  (5.5, 1) -- (5.5, -1) node[label=below:$L_2$]{};
\end{tikzpicture}}
\subfloat[][]{
\begin{tikzpicture}[style=thick, scale=.8]
\node (1) at (0, 2) [label=above:1]{};
\node (2) at (-2,0) [label=left:2]{};
\node (3) at (-.5, .5) [label=left:3]{};
\node (4) at (2,0) [label=right:4]{};
\node (5) at (0, -.5) [label=below:5]{};

\filldraw [black] (1) circle (2pt) 
(2) circle (2pt) 
(3) circle (2pt)
(4) circle (2pt)
(5) circle (2pt);

\draw (0,2) -- (0,-.5) -- (-2,0) --cycle;
\draw (0,2) --(2,0);
\draw (-2,0) -- (0,-.5) -- (2,0);
\draw (0,2) -- (-.5,.5) -- (0,-.5);
%\draw  (5.5, 1) -- (5.5, -1) node[label=below:$L_2$]{};
\end{tikzpicture}}
% %\end{minipage}
\caption{The three bases of the rigidity matroid when $d=2$ and $n=5.$}
\label{fig: bases25}
\end{figure}
\noindent 
Adding an edge to any of these graphs creates a circuit. 

What about the bases for 
arbitrary $n$ and $d$?  
We can derive a necessary condition using an idea of Maxwell \cite{maxwell}.
The dimension of $\CM_{d,n}$ is 
$dn - \binom{d+1}{2}$ (see \cite{borcea}), 
so no independent set $G$ in the algebraic matroid 
$(K_n,\cI_{\CM_{d,n}})$ can contain more than this many edges, since 
the dimension of $\pi_G$ is bounded by that of $\CM_{d,n}$.  The same 
argument applies to any induced subgraph of $K_n$,
since the projection of $\CM_{d,n}$ onto a 
smaller $K_{n'}$ is $\CM_{d,n'}$,so any basis graph must 
have $dn - \binom{d+1}{2}$ edges 
and no induced subgraph on $n'$ vertices with 
more than $\max\{0,dn' - \binom{d+1}{2}\}$ edges.  
Such a graph
is called $\left(d,\binom{d+1}{2}\right)$-tight.  

The following theorem is usually attributed to Laman \cite{laman}, but see 
also Pollaczek-Geiringer \cite{hilda}\footnote{Jan Peter Schäfermeyer brought Pollaczek-Geiringer's
work to the attention of the framework rigidity community in 2017.}.
\begin{thm}[Laman's Theorem]\label{thm: laman}
For all $n\ge 2$, the bases of the rigidity matroid $(K_n,\cI_{\CM_{2,n}})$
are the $(2,3)$-tight graphs.
\end{thm}
Aside from dimension one, which is folklore (the bases are spanning trees of $K_n$),
and $n\le d+2$, which gives a uniform matroid, there is no
known analogue of Laman's Theorem in higher dimensions.  Finding one
is a major open problem in rigidity theory.  
In dimensions $d\ge 3$ Maxwell's heuristic no longer 
rules out all the circuits in the rigidity matroid.  An interesting class 
of examples was constructed by Bolker and Roth \cite{bolker}.  
They showed that, for $d\ge 3$, $K_{d+2,d+2}$ is a circuit in the 
rigidity matroid with $2(d+2)$ vertices and $(d+2)^2$ edges.  Since 
\[dn-\binom{d+1}{2} - (d+2)^2 = 2d(d+2) - \binom{d+1}{2} - (d+2)^2 = \frac{1}{2} \left(d^2-d-8\right) >0\]
when $d\ge 4$, Maxwell's heuristic fails on $K_{d+2,d+2}$ for $d=4$ and becomes
less effective as $d$ increases.

\section{Final thoughts.}
As we have seen, the perspective of matroid theory reveals a beautiful interplay among objects that are
connected in spirit if different
in origin.  Furthermore, there is much yet to explore on both the computational and theoretical sides.

A type of question that is particularly relevant in applications is computational in nature.
%how to compute the algebraic data associated with a specific realization of a 
%known algebraic matroid.  
We don't know a general method other than elimination to 
compute circuit polynomials.  %, even in the matroids arising from parameterized 
%varieties.  
As an example, the circuit polynomial of $K_{3,4}$ in the $2$-dimensional
rigidity matroid seems out of reach to naive implementation in current computer 
algebra systems, despite having a simple geometric description, by White and Whiteley \cite{WW} 
in the coordinates of the joints.  To this end, Rosen \cite{Rosen}, has developed software that combines 
linear algebra and numerical algebraic geometry to speed up computation in algebraic 
matroids that have additional geometric information.

Additionally, a number of basic structural questions about algebraic matroids remain unresolved.  
Strikingly, it is not even known if the class of algebraic matroids is closed under duality (see \cite[Section 6.7]{oxley}).  
Enumerative results are also largely unavailable.  Nelson's recent breakthrough \cite{nelson} 
shows that almost all matroids are not linear, which in light of Ingleton's 
Theorem \ref{thm: linear} implies the same thing about algebraic matroids in characteristic zero.  
It would be interesting to know if similar results hold for algebraic matroids in positive 
characteristic.

    \bigskip
    
\noindent \emph{Acknowledgements}  
The first and third authors wish to thank Franz Király for many helpful conversations
during previous projects which have influenced their understanding of algebraic
matroids. We also wish to thank Bernd Sturmfels and David Cox for their encouragement, Will Traves for
helpful conversations, and Dustin Cartwright for comments on the Lindström  valuation.

\end{document}